\documentclass[12pt]{article}
\usepackage{latexsym, amsmath, amsfonts, amsthm, amssymb}
\usepackage{times}
\usepackage{a4wide}
\usepackage{tikz,pgfplots}
\usepackage{mathrsfs}
\usetikzlibrary{arrows}

\def \RR {\mathbb R}

\def \ZZ {\mathbb Z}

\def \DD {\mathbb D}
\def \MBS {\mathbb S}

\def \eps {\varepsilon}

\newtheorem{theorem}{Theorem}[section]
\newtheorem{definition}[theorem]{Definition}

\newtheorem{lemma}[theorem]{Lemma}

\newtheorem{proposition}[theorem]{Proposition}

\def\myffrac#1#2 in #3{\raise 2.6pt\hbox{$#3 #1$}\mkern-1.5mu\raise 0.8pt\hbox{$
		#3/$}\mkern-1.1mu\lower 1.5pt\hbox{$#3 #2$}}

\def\qed{\hfill $\vcenter{\hrule height .3mm
		\hbox {\vrule width .3mm height 2.1mm \kern 2mm \vrule width .3mm
			height 2.1mm} \hrule height .3mm}$ \bigskip}

\begin{document}

\title{A quantitative approach to the regularity of a Riemannian surface}
\author{Matan Eilat}
\date{}
\maketitle

\begin{abstract}
We introduce two definitions with the purpose of quantifying the concept of a $C^{2,\alpha}$ surface for $0 < \alpha < 1$.
The intrinsic definition is given in terms of the $\alpha$-H\"{o}lder norm of the Gauss curvature function. 
The extrinsic one relies on the existence of a smooth local representation of the Riemannian metric. 
We show that these definitions are equivalent up to a constant depending on $\alpha$.
\end{abstract}

\section{Introduction}

The simplest geometric spaces in two dimensions are probably the simply-connected manifolds of constant Gauss curvature: The sphere, the hyperbolic plane and the Euclidean plane.
Here we consider a more inclusive class of surfaces in which the curvature is not necessarily constant, yet it does not vary too wildly. More precisely, we are concerned with Riemannian surfaces whose Gauss curvature function is $\alpha$-H\"{o}lder continuous for some $0 < \alpha < 1$.
According to DeTurck and Kazdan \cite{DK} any such surface is $C^{2,\alpha}$-smooth, meaning that each point is covered by some coordinate chart in which the metric takes the form
$$
g = \sum_{i,j=1}^{2}g_{ij}dx^{i}dx^{j} ,
$$
where $g_{ij}$ are $C^{2,\alpha}$-smooth functions.
We quantify this statement and its converse by giving two definitions of a surface being {\it $\alpha$-regular} at a given point to a certain distance.

\medskip
Let $M$ be a $C^{2}$-smooth Riemannian surface and $0< \alpha \leq 1$. We write $d_{g}$ for the Riemannian distance in $M$, and $K:M \rightarrow \RR$ for the Gauss curvature function.
For a point $p_{0} \in M$ and $\delta > 0$ we write $B = B_{M}(p_{0}, \delta)$ for the Riemannian disc of radius $\delta$ around $p_{0}$. 
We define the {\it non-dimensional} $C^{0,\alpha}$-norm of $K$ in $B$  by
$$
\left\Vert K \right\Vert_{0,\alpha; B} = 
\sup_{p \in B} |K(p)|  + 
(2\delta)^{\alpha} \sup_{p \neq q \in B}\frac{|K(p)-K(q)|}{d_{g}^\alpha(p,q)}.
$$

\begin{definition} 
We say that $M$ is intrinsically $\alpha$-regular at a point $p_{0} \in M $ to distance $\delta > 0$ if the injectivity radius of all points of $B = B_{M}(p_{0},\delta)$ is at least $2\delta$, and 
$$
\delta^{2}\cdot \left\Vert K \right\Vert_{0,\alpha;B} \leq 1.
$$
We define the intrinsic $\alpha$-regularity radius of $M$ at a point $p_{0} \in M$ to be the supremum of all such $\delta$, and denote it by $\rho_{\text{int}} = \rho_{\text{int}}(p_{0},\alpha)$.
\label{def_1.1}
\end{definition}

\medskip
For an open set $U \subseteq \RR^{2}$ we write $g_{0} = \sum_{i,j=1}^{2} \delta_{ij}dx^{i}dx^{j}$ for the Euclidean metric on $U$, where $\delta_{ij}$ stands for the Kronecker delta function.
We write $d_{g_{0}}$ for the Euclidean distance in $U$. For a non-negative integer $k$ and a function $f \in C^{k}(U)$ we set
$$
\left\Vert D^{k}f \right\Vert_{0;U} = 
\max_{|\beta| = k} \left\Vert D^{\beta}f \right\Vert_{0;U} =
\max_{|\beta| = k} \sup_{x \in U} \left| D^{\beta}f(x)\right|,
$$
and
$$	
\left[ D^{k}f \right]_{\alpha;U} = 
\max_{|\beta| = k} \left[ D^{\beta}f \right]_{\alpha;U} =
\max_{|\beta| = k} \sup_{x \neq y \in U} \frac{\left| D^{\beta}f(x) - D^{\beta}f(y)\right|}{d^{\alpha}_{g_{0}}(x,y)},
$$
where $\beta$ is a two-dimensional multi-index. Assuming that $U$ is bounded with 
$d = \text{diam}(U) = \sup_{x,y \in U} d_{g_{0}}(x,y)$
we define the non-dimensional norms
$$
\left\Vert f \right\Vert_{k;U} = 
\sum_{j=0}^{k} d^{j} \left\Vert D^{j}f \right\Vert_{0;U}
\qquad \text{and} \qquad
\left\Vert f \right\Vert_{k,\alpha;U} =
\left\Vert f \right\Vert_{k;U} + 
d^{k+\alpha} \left[ D^{k}f \right]_{\alpha;U}.
$$

\begin{definition} 
We say that $M$ is extrinsically $\alpha$-regular at a point $p_{0}\in M$ to distance $\delta > 0$ if the injectivity radius of all points of $B = B_{M}(p_{0},\delta)$ is at least $2\delta$, and $B$ is isometric to a bounded open set $U \subseteq \RR^{2}$ endowed with the Riemannian metric 
$ g = \sum_{i,j=1}^{2}g_{ij}dx^{i}dx^{j}$, 
such that $g_{ij}:U \rightarrow \RR$ satisfy 
\begin{equation}
g_{ij}(x_{0}) = \delta_{ij}
\qquad \text{and} \qquad
\left\Vert g_{ij} - \delta_{ij} \right\Vert_{2,\alpha;U} \leq 1 \qquad \text{for any }\;i,j\in \{ 1,2 \},
\label{eq_1829}
\end{equation}
where $x_{0} \in U$ is the image of $p_{0}$ under the implied isometry.
We define the extrinsic $\alpha$-regularity radius of $M$ at a point $p_0\in M$ to be the supremum of all such $\delta$, and denote it by $\rho_{\text{ext}} = \rho_{\text{ext}}(p_{0},\alpha)$.
\label{def_2345}
\end{definition}

Theorem \ref{thm_1} and Theorem \ref{thm_2} below imply that the two definitions for the $\alpha$-regularity radius are equivalent up to a constant depending on $\alpha$.
Loosely speaking, this means that it is possible to switch between the intrinsic and the extrinsic points of view while only scaling the radius by some constant factor.

\begin{theorem}
Let $M$ be a $C^{2}$-smooth Riemannian surface, let $p_{0}\in M$ and $0< \alpha \leq 1$. Then
$$
\rho_{\text{int}} \geq C_{1} \cdot \rho_{\text{ext}},
$$
for some universal constant $C_{1} > 0$.
\label{thm_1}
\end{theorem}

\begin{theorem}
Let $M$ be a $C^{2}$-smooth Riemannian surface, let $p_{0}\in M$ and $0< \alpha < 1$. Then
$$
\rho_{\text{ext}} \geq C_{2} \cdot \rho_{\text{int}},
$$
for some constant $C_{2} > 0$ depending on $\alpha$.
\label{thm_2}
\end{theorem}

\medskip
Note that in Theorem \ref{thm_1} we allow $\alpha = 1$.
However, in order to prove Theorem \ref{thm_2} we use tools from harmonic analysis which do not apply for $\alpha = 1$. In particular, as part of the proof we use an {\it isothermal coordinate chart}. It therefore follows from Theorem \ref{thm_1} and Theorem \ref{thm_2} that if the metric satisfies certain $C^{2,\alpha}$-estimates in some coordinate chart on a disc of radius $\delta > 0$, then it satisfies the same $C^{2,\alpha}$-estimates in an isothermal coordinate chart on a disc of radius $C \delta$ for some constant $C > 0$ depending on $\alpha$.
This fact may be thought of as a quantitative evidence for the superior regularity of isothermal coordinates, as shown in \cite{DK}.

\medskip
The above definitions have a natural global extension: we say that a $C^{2}$-smooth Riemannian surface $M$ is $\alpha$-regular if 
$ \rho(M,\alpha) = \inf_{p \in M} \rho(p,\alpha) > 0$, and define 
$$
\Vert M \Vert_{\alpha} = \rho^{-2}(M,\alpha).
$$
One may choose to use either the intrinsic or the extrinsic point of view, as they are equivalent according to Theorem \ref{thm_1} and Theorem \ref{thm_2}. It is not too difficult to see that $\Vert M \Vert_{\alpha} = 0$ if and only if $M$ is isometric to the Euclidean plane, and that $\Vert \cdot \Vert_{\alpha}$ is homogeneous of degree $1$ with respect to a scale of the Riemannian metric.
We hope that this quantitative approach to the $\alpha$-regularity of a Riemannian surface can be useful for several research directions, for example in the areas of surface fitting and approximation theory.

\medskip
In the $n$-dimensional case there is a notion of a similar flavour to ours called the {\it harmonic radius}. 
Roughly speaking, the $C^{k,\alpha}$-harmonic radius at $p_{0} \in M$ is the largest radius $\delta > 0$ such that there exists a harmonic coordinate chart on the geodesic ball $B_{M}(p_{0},\delta)$ in which the metric tensor is $C^{k,\alpha}$-controlled. Lower bounds on the $C^{1,\alpha}$-harmonic radius in terms of the injectivity radius and the Ricci curvature were given by Anderson in the context of Gromov-Hausdorff convergence of manifolds \cite{An}. For other estimates and more details we refer to \cite{HH, ZZ}.

\medskip
{\it Acknowledgements.} I would like to express my deep gratitude to my advisor, Prof. Bo'az Klartag, for his continuous guidance, assistance and encouragement.
Supported by the Adams Fellowship Program of the Israel Academy of Sciences and Humanities and the Israel Science Foundation (ISF).

\section{Curvature estimates via metric coefficients}
\label{sec_claim_1}

\medskip
The first step in the proof of Theorem \ref{thm_1} is to allow ourselves to work with metrics whose coefficients satisfy tighter bounds on their $C^{2,\alpha}$-norm.
For this purpose we introduce the following auxiliary definition which replaces the bound in condition (\ref{eq_1829}) from Definition \ref{def_2345} by some $\eps > 0$. We will not use arbitrary values of $\eps$, but merely a constant value which is small enough for our needs. Hence from now on we fix
$$
\eps = 0.01.
$$
We use the Einstein summation convention, and all the indices are taken implicitly from the set $\{1,2\}$ unless stated otherwise. With a slight abuse of notation, we sometimes do not distinguish between a point $p \in B$ and its image in $U$ under the implied isometry.

\begin{definition}
We say that $M$ is $\alpha$-regular via an $\eps$-restricted coordinate chart at a point $p_{0}\in M$ to distance $\delta > 0$ if the injectivity radius of all points of $B = B_{M}(p_{0},\delta)$ is at least $2\delta$, and $B$ is isometric to a bounded open set $U \subseteq \RR^{2}$ endowed with the Riemannian metric 
$ g = g_{ij}dx^{i}dx^{j}$, such that $g_{ij}:U \rightarrow \RR$ satisfy 
$$
g_{ij}(p_{0}) = \delta_{ij} 
\qquad \text{and} \qquad
\left\Vert g_{ij} - \delta_{ij} \right\Vert_{2,\alpha;U} \leq \eps \qquad \text{for all }\;i,j\in \{ 1,2 \}.
$$
We define the $\alpha$-regularity radius via an $\eps$-restricted coordinate chart of $M$ at a point $p_0\in M$ to be the supremum over all such $\delta$, and denote it by $\rho_{\eps-\text{chart}} = \rho_{\eps-\text{chart}}(p_{0},\alpha)$.
\label{def_2345_scaled}
\end{definition} 

\medskip
Let $g = g_{ij}dx^{i}dx^{j}$ be a Riemannian metric on an open set $U \subseteq \RR^{2}$. 
When considering a certain point $x \in U$ we write $\Vert \cdot \Vert_{g} = \sqrt{g(\cdot,\cdot)}$ for the Riemannian norm on the tangent space $T_{x}U$. 
We denote the Riemannian length of a rectifiable path $\gamma \subseteq U$ by $L_{g}(\gamma)$.
Suppose that $(U,g)$ is isometric to the Riemannian disc $B = B_{M}(p_{0},\delta)$.
The radius $\delta$ of the disc $B$ can then be related to the Euclidean diameter $d$ of the set $U$ in terms of 
$c = \max_{i,j} \Vert g_{ij} - \delta_{ij} \Vert_{0;U}$ as follows:
Let $q \in \partial{U}$ be such that 
$d_{g_{0}}(p_{0}, q) = \inf_{p \in \partial{U}} d_{g_{0}}(p_{0}, p)$, where $d_{g_{0}}$ stands for the Euclidean distance in $U$.
Let $\gamma(t)=(1-t)p_{0} + tq$ be the line connecting $p_{0}$ and $q$.
Then $\gamma([0,1)) \subseteq U$ and 
\begin{equation}
	\delta = 
	d_{g}(p_{0},q) \leq
	L_{g}(\gamma) \leq 
	\sqrt{1+2c} \cdot L_{g_{0}}(\gamma) = 
	\sqrt{1+2c} \cdot d_{g_{0}}(p_{0},q) \leq 
	\sqrt{1+2c} \cdot d.
	\label{eq_delta_d}
\end{equation}

\medskip
The following proposition implies that in order to show Theorem \ref{thm_1} it suffices to consider a metric which satisfies the assumptions of Definition \ref{def_2345_scaled}. 

\begin{proposition}
Let $M$ be a $C^{2}$-smooth Riemannian surface, let  $p_{0}\in M$ and $0< \alpha \leq 1$. Then
$$
\rho_{\eps-\text{chart}} \geq \frac{\eps}{14} \cdot \rho_{\text{ext}}.
$$
\label{prop_chart_def_scale}
\end{proposition}

\begin{proof}
Suppose that $M$ is extrinsically $\alpha$-regular at $p_{0} \in M$ to distance $\delta > 0$.
For $q \in \partial{U}$ let $\gamma_{q} : [0, \delta] \to U$ be the minimizing geodesic with respect to the metric $g$ such that $\gamma_{q}(0) = p_{0}$ and $\gamma_{q}(\delta) = q$.
Since $\max_{i,j} \Vert g_{ij} - \delta_{ij} \Vert_{0;U} \leq 1$ we have that 
$ L_{g_{0}}(\gamma) \geq \delta / \sqrt{3}$. Hence there exists $0 < t_{q} < \delta$ such that 
\begin{equation}
L_{g_{0}}(\gamma_{q}([0,t_{q}])) = \frac{\eps\delta} {8\sqrt{3}} \leq \frac{\eps d}{8},
\label{eq_1105}
\end{equation}
where $d$ is the Euclidean diameter of $U$, and the inequality follows from (\ref{eq_delta_d}).
According to our assumptions we have that 
$ \max_{i,j}\Vert D^{1}g_{ij} \Vert_{0;U} \leq d^{-1} $.
Thus by the mean value theorem, for any $x \in \gamma_{q}([0,t_{q}])$ we have
\begin{equation}
\max_{i,j} |g_{ij}(x) - \delta_{ij}| \leq \sqrt{2} d^{-1} \cdot \frac{\eps d}{8} = \frac{\eps}{4\sqrt{2}}.
\label{eq_1755}
\end{equation}

\medskip
It follows from (\ref{eq_1105}) and (\ref{eq_1755}) that  
$ L_{g}(\gamma_{q}([0,t_{q}])) \geq 
0.99 \cdot L_{g_{0}}(\gamma_{q}([0,t_{q}]) \geq \eps \delta /14$.
The Riemannian disc $B_{M}( p_{0}, \eps \delta / 14 )$ is therefore isometric to a set $V \subseteq \bigcup_{q \in \partial{U}} \gamma_{q}([0,t_{q}]) $ endowed with the metric $g$. 
Moreover, it follows from (\ref{eq_1105}) that  $d_{g_{0}}(p_{0}, \gamma_{q}(t_{q})) \leq \eps d / 8$ for any $q \in \partial{U}$. Writing $d_{V}$ for the Euclidean diameter of $V$ we obtain that $d_{V} \leq \eps d/4$. 
Thus for any $i,j \in \{1,2\}$,
$$
d_{V}^{k} \left\Vert D^{k}g_{ij} \right\Vert_{0;V} \leq (\eps / 4)^{k} 
\quad \text{for any } k\in \{1,2\}
\qquad \text{and} \qquad
d_{V}^{2+\alpha}[D^{2}g_{ij}]_{\alpha;V} \leq (\eps / 4)^{2+\alpha}.
$$
Furthermore, using (\ref{eq_1755}) we see that $ \max_{i,j} \Vert g_{ij}-\delta_{ij} \Vert_{0;V} \leq \eps / 4$, so that for any $i,j \in \{1,2\}$,
$$
\left\Vert g_{ij} - \delta_{ij} \right\Vert_{2,\alpha;V} = \left( \sum_{k=0}^{2}d_{V}^{k} \left\Vert D^{k}(g_{ij}-\delta_{ij}) \right\Vert_{0;V} \right) + d_{V}^{2+\alpha}[D^{2}g_{ij}]_{\alpha;V} \leq \epsilon,
$$
and the proof is completed.

\end{proof}

\medskip
From now on and until the end of this section, we fix $0 < \alpha \leq 1$ and assume that $M$ is $\alpha$-regular via an $\eps$-restricted coordinate chart at a point $p_{0} \in M$ to distance $\delta > 0$. We follow the notation of Definition \ref{def_2345_scaled}, and write $d$ for the Euclidean diameter of $U$.
The following lemma implies that under these assumptions the  Riemannian disc $B$ is strongly convex, i.e. for any two points $p$ and $q$ in the closure $\overline{B}$ there is a unique minimizing geodesic in $M$ from $p$ to $q$, and the interior of this geodesic is contained in $B$. 

\begin{lemma} 
We have that $\delta^{2} \sup_{B} | K | < 0.1$ and $B$ is strongly convex.
\label{lem_delta_curv_bounded}
\end{lemma}

\begin{proof}
In order to estimate the Gauss curvature function we use the following expression:
\begin{equation}
	K = -\frac{1}{g_{11}} \left( 
	\frac{\partial}{\partial{x_{1}}} \Gamma^{2}_{12} - 
	\frac{\partial}{\partial{x_{2}}} \Gamma^{2}_{11} +
	\Gamma^{1}_{12}\Gamma^{2}_{11} - 
	\Gamma^{1}_{11}\Gamma^{2}_{12} + 
	\Gamma^{2}_{12}\Gamma^{2}_{12} - 
	\Gamma^{2}_{11}\Gamma^{2}_{22}
	\right) ,
	\label{eq_gauss_curvature_metric}
\end{equation}
where $\Gamma_{ij}^{k}$ are the Christoffel symbols of the second kind. See \cite{St} for more details. 
For the evaluation of the Christoffel symbols we use
\begin{equation}
	\Gamma^{k}_{ij} = \frac{1}{2} \, g^{lk}(\partial_{j}g_{il} + \partial_{i}g_{jl} - \partial_{l}g_{ji}),
	\label{eq_Christoffel}
\end{equation} 
where $[ g^{ij} ]$ denotes the inverse matrix of $[ g_{ij} ]$.
The norms in this proof are implicitly with respect to the set $U$.
Since $\det g = g_{11}g_{22} - g_{12}g_{21} \geq 1-2\eps $ we see that
$$
\max_{i,j} \left\Vert g^{ij} \right\Vert_{0} \leq 
\max_{i,j} \left\Vert g_{ij} \right\Vert_{0} \cdot 
\left\Vert \frac{1}{\det g} \right\Vert_{0} \leq 
\frac{1+\eps}{1-2\eps}.
$$
Thus according to (\ref{eq_Christoffel}) together with the assumption that $\max_{i,j} \left\Vert D^{1}g_{ij} \right\Vert_{0} \leq \eps d^{-1}$ we have
\begin{equation}
\max_{i,j,k} \left\Vert \Gamma_{ij}^{k} \right\Vert_{0} \leq 
\max_{i,j} \left\Vert g^{ij} \right\Vert_{0} \cdot 
3 \max_{i,j} \left\Vert D^{1}g_{ij} \right\Vert_{0}
\leq
\frac{1+\eps}{1-2\eps} \cdot 3\eps d^{-1} < \frac{d^{-1}}{32}.
\label{eq_bound_Christoffel}
\end{equation}
Since $\partial_{l}g^{ij} = -g^{ia}g^{jb}\partial_{l}g_{ab}$ we see that
$$
\max_{i,j} \left\Vert D^{1}g^{ij} \right\Vert_{0} 
\leq 
4 \cdot \left( \max_{i,j } \left\Vert g^{ij} \right\Vert_{0} \right)^2 \cdot \max_{i,j} \left\Vert D^{1}g_{ij} \right\Vert_{0}
\leq  
4 \left( \frac{1+\eps}{1-2\eps} \right)^{2} 
\eps d^{-1}.
$$
Differentiating (\ref{eq_Christoffel}) and using the assumption that $\max_{i,j} \left\Vert D^{2}g_{ij} \right\Vert_{0} \leq \eps d^{-2}$ we obtain that
$$
\max_{i,j,k} \left\Vert D^{1}\Gamma_{ij}^{k} \right\Vert_{0} \leq
\max_{i,j} \left\Vert D^{1}g^{ij} \right\Vert_{0} \cdot
3 \max_{i,j} \left\Vert D^{1}g_{ij} \right\Vert_{0} +
\max_{i,j} \left\Vert g^{ij} \right\Vert_{0} \cdot 
3 \max_{i,j} \left\Vert D^{2}g_{ij} \right\Vert_{0} < 
\frac{d^{-2}}{31}.
$$
By (\ref{eq_gauss_curvature_metric}) together with the above estimates and the assumption that $g_{11} \geq 1-\eps$ we obtain that
$$
d^{2} \left\Vert K \right\Vert_{0} \leq 
\left\Vert \frac{1}{g_{11}} \right\Vert_{0} \cdot 
\left[ 2 d^{2} \max_{i,j,k} \left\Vert D^{1}\Gamma_{ij}^{k} \right\Vert_{0} + 
4 \left( d \max_{i,j,k} \left\Vert \Gamma_{ij}^{k} \right\Vert_{0}\right)^{2} 
\right] < \frac{1}{14}.
$$

\medskip
According to (\ref{eq_delta_d}) we have $\delta \leq \sqrt{1+2\eps} \cdot d$ and therefore $\delta^{2} \sup_{B} |K| < 0.1$. 
Under our assumption that the injectivity radius at each point of the Riemannian disc $B$ is at least $2\delta$, we obtain that $B$ is strongly convex according to the Whitehead theorem \cite{Wh}.
\end{proof}

\medskip
Our next lemma states that the set $U$ is convex in the Euclidean sense, i.e. as a subset of $\RR^{2}$ with respect to the Euclidean metric.

\begin{lemma}
The set $U$ is convex in the Euclidean sense.
\label{lem_U_convex}
\end{lemma}

\begin{proof}
We show that $U$ is convex by showing that the boundary of $U$ is a convex curve. Since the $C^{2}$-norms of the functions $g_{ij}$ are bounded, we may extend $g_{ij}$ and their derivatives continuously to the closure of $U$, and therefore speak of the inner product $g(\cdot,\cdot) : T_{q}\RR^{2} \times T_{q}\RR^{2} \to \RR$ and the Christoffel symbols $\Gamma_{ij}^{k}$ at a point $q \in \partial{U}$.  

\medskip
Fix $q \in \partial{U}$ and let $\gamma : (-\tau, \tau) \to \partial{U}$ be a unit speed curve with respect to $g$ such that $\gamma(0)=q$.
Let $N \in T_{q}\RR^{2}$ be the inward unit normal at $q$ with respect to $g$, i.e. 
$$
\left\Vert N \right\Vert_{g} = 1, \qquad 
g(N, \dot{\gamma}) = 0, \qquad \text{and} \qquad
\{ \exp_{q}(tN) : t\in (0,t_{0})\} \subseteq U \text{ for some } t_{0}>0,
$$
where all the expressions are evaluated at $q$. 
Similarly, let $N_{0}$ be the inward unit normal at $q$ with respect to the Euclidean metric $g_{0}$.
The set $ T_{q}\RR^{2} \setminus \{ t \cdot \dot{\gamma}(0): t \in \RR \}$ consists of two connected components, one of which may be characterized by the sets
$ \{ u \in T_{q}\RR^{2} :g(u,N) > 0 \} = \{ u \in T_{q}\RR^{2} :g_{0}(u,N_{0}) > 0 \} $.
In particular we have that
\begin{equation}
g(N,N_{0}) > 0 \qquad \text{and} \qquad g_{0}(N,N_{0}) > 0.
\label{eq_normal_in_prod_pos}
\end{equation}

\medskip
Write $D_{v}u$ for the covariant derivative with respect to $g_{0}$, and $\nabla_{v}u$ for the covariant derivative with respect to $g$.
In order to show that $U$ is convex in the Euclidean sense it suffices to show that
\begin{equation}
g_{0}(D_{\dot{\gamma}}\dot{\gamma},N_{0}) > 0.
\label{eq_0129}
\end{equation}
Since $B$ is a Riemannian disc and $\gamma$ is a unit speed curve with respect to $g$, the Riemannian acceleration $\nabla_{\dot{\gamma}}\dot{\gamma}$ may be expressed in terms of the Laplacian of the distance function $r$ from the center of the disc $p_{0}$, so that
$$
\nabla_{\dot{\gamma}}\dot{\gamma}(q) = \Delta_{M}r(q) \cdot N,
$$
where $\Delta_{M}$ stands for the Laplace-Beltrami operator. Hence by the triangle inequality,
\begin{equation}
\left| g_{0}(D_{\dot{\gamma}}\dot{\gamma},N_{0}) - \Delta_{M}r \right| \leq
\left| g_{0}(D_{\dot{\gamma}}\dot{\gamma} - \nabla_{\dot{\gamma}}\dot{\gamma},N_{0}) \right| +
\Delta_{M}r \left|g_{0}(N, N_{0}) - 1 \right|,
\label{eq_0123}
\end{equation}
where all the expressions are evaluated at $q$.
Since $\gamma$ is a unit speed curve with respect to $g$ we have that $\Vert \dot{\gamma} \Vert_{g_{0}}^{2}\leq 1/(1-2\eps)$. 
Using the bound we obtained for the supremum norm of the Christoffel symbols in (\ref{eq_bound_Christoffel}) together with (\ref{eq_delta_d}) we see that $\max_{i,j,k} \Vert \Gamma_{ij}^{k} \Vert_{0} \leq \delta^{-1}/31$. Thus
$$
\max_{k} \left| \dot{\gamma}^{i}\dot{\gamma}^{j}\Gamma_{ij}^{k} \right| \leq 
\max_{i,j,k} \left\Vert \Gamma_{ij}^{k} \right\Vert_{0} \cdot (|\dot{\gamma}^{1}| + |\dot{\gamma}^{2}|)^{2} \leq 
2 \cdot \max_{i,j,k} \left\Vert \Gamma_{ij}^{k} \right\Vert_{0} \cdot 
\left\Vert \dot{\gamma} \right\Vert_{g_{0}}^{2} \leq 
\frac{\delta^{-1}}{15}.
$$
Since
$
\nabla_{\dot{\gamma}}\dot{\gamma} - D_{\dot{\gamma}}\dot{\gamma} = (\dot{\gamma}^{i}\dot{\gamma}^{j}\Gamma_{ij}^{k})\partial_{k}
$ 
we see that
$ \Vert \nabla_{\dot{\gamma}}\dot{\gamma} - D_{\dot{\gamma}}\dot{\gamma}\Vert_{g_{0}} \leq \sqrt{2}\delta^{-1}/15$.
Let $\kappa = \sup_{B} |K|$. It follows from a well-known comparison argument that the Laplacian of the distance function satisfies
$ \sqrt{\kappa} \cot(r \sqrt{\kappa}) \leq \Delta_{M} r \leq \sqrt{\kappa} \coth(r \sqrt{\kappa})$, see \cite{Pet} for more details.
In particular, since $\delta \sqrt{\kappa} < 1/3$ by Lemma  \ref{lem_delta_curv_bounded}, we have that
$ \delta \cdot \Delta_{M}r(q) \geq \delta \sqrt{\kappa} \cot(\delta \sqrt{\kappa}) \geq 0.96 $.
It therefore follows from the Cauchy-Schwarz inequality that
\begin{equation}
\frac{|g_{0}(\nabla_{\dot{\gamma}}\dot{\gamma} - D_{\dot{\gamma}}\dot{\gamma}, N_{0})|}{\Delta_{M}r} \leq 
\frac{\Vert \nabla_{\dot{\gamma}}\dot{\gamma} - D_{\dot{\gamma}}\dot{\gamma}\Vert_{g_{0}}}{\Delta_{M}r}  \leq
\frac{\sqrt{2}}{15} \cdot \frac{1}{\delta \Delta_{M}r} \leq 0.1,
\label{eq_0123_2}
\end{equation}
where all the expressions are evaluated at $q$.
Let $v = \dot{\gamma}/\Vert \dot{\gamma}\Vert_{g_{0}} \in T_{q}\RR^{2}$.
Using the assumption that $\max_{i,j} \Vert g_{ij} - \delta_{ij} \Vert \leq \eps$ we obtain that
$$
|g_{0}(N,v)| = |g_{0}(N,v) - g(N,v)| \leq 
2 \eps \left\Vert N \right\Vert_{g_{0}} \left\Vert v \right\Vert_{g_{0}}  \leq
 2\eps / \sqrt{1-2\eps}.
$$
Moreover, since $v$ and $N_{0}$ are orthonormal with respect to $g_{0}$ we have that
$$
\frac{97}{100} < \frac{1}{1+2\eps} - \frac{4\eps^{2}}{1-2\eps} \leq \Vert N \Vert_{g_{0}}^{2} - |g_{0}(N,v)|^{2} = |g_{0}(N,N_{0})|^{2} \leq \Vert N \Vert_{g_{0}}^{2} \leq \frac{1}{1-2\eps} = \frac{100}{98}.
$$
Since $g_{0}(N,N_{0}) > 0$ according to (\ref{eq_normal_in_prod_pos}), it follows from (\ref{eq_0123}) and (\ref{eq_0123_2}) that 
$$
\left| \frac{g_{0}(D_{\dot{\gamma}}\dot{\gamma},N_{0})}{\Delta_{M} r} - 1 \right| < 1.
$$
Therefore (\ref{eq_0129}) holds true and the proof is completed.
\end{proof}

\begin{proposition}
There exists a universal constant $C$ such that $\delta^{2} \left\Vert K \right\Vert_{0,\alpha;B} \leq C$.
\label{prop_bound_curv_alpha_norm}
\end{proposition}

\begin{proof} 
For convenience we write $x \lesssim y$ if there exists a universal constant $C$ such that $x \leq C \cdot y$. In case $x \lesssim y$ and $y \lesssim x$ we write $x \eqsim y$.
According to Lemma \ref{lem_delta_curv_bounded} we have that $B$ is strongly convex. Lemma \ref{lem_U_convex} shows that $U$ is convex in the Euclidean sense. 
It therefore follows from the assumption $\max_{i,j} \Vert g_{ij}-\delta_{ij}\Vert_{0;U} \leq 0.01$ that for any $p,q \in U$,
$$
0.98 \cdot d_{g_{0}}(p,q)\leq d_{g}(p,q) \leq 1.01 \cdot d_{g_{0}}(p,q).
$$
Thus by considering $\Vert K \Vert_{0,\alpha;U}$ in the Euclidean sense, i.e. where distances are with respect to the Euclidean metric $g_{0}$, we only lose a constant factor.
Therefore all of the norms in this proof are implicitly with respect to $U$, and our goal is to show that 
\begin{equation}
d^{2} \left\Vert K \right\Vert_{0,\alpha} \lesssim 1.
\label{eq_1327}
\end{equation}
Since $U$ is convex we have the inclusion of H\"{o}lder spaces $C^{k+1}(\overline{U}) \subseteq C^{k,\alpha}(\overline{U})$. More precisely, the mean value theorem implies that $[f]_{\alpha} \leq \sqrt{2} d^{1-\alpha} \Vert D^{1}f \Vert_{0}$ for any $f \in C^{1}(\overline{U})$. It therefore follows from our assumptions that
$$
\left\Vert g_{ij} \right\Vert_{0,\alpha} \leq \delta_{ij} + \sqrt{2} \eps
\qquad \text{and} \qquad 
d \left\Vert D^{1}g_{ij} \right\Vert_{0,\alpha} \leq \sqrt{2} \eps,
$$
where $\Vert D^{1} g_{ij} \Vert_{0,\alpha} = \sup_{|\beta| = 1} \Vert D^{\beta} g_{ij} \Vert_{0,\alpha}$ and $\beta$ is a two-dimensional multi-index.
In the following estimates we implicitly use the sub-multiplicativity of $\Vert \cdot \Vert_{0,\alpha}$ and the fact that 
$\Vert 1 / f \Vert_{0,\alpha} \leq
\Vert 1 / f \Vert_{0} + 
\Vert f \Vert_{0,\alpha} \cdot 
\Vert 1 / f \Vert_{0}^{2}$ for a non-vanishing function $f$. We have that
$$
\max_{i,j} \left\Vert g^{ij} \right\Vert_{0,\alpha} \leq 
\max_{i,j} \left\Vert g_{ij} \right\Vert_{0,\alpha} \cdot 
\left\Vert \frac{1}{\det g} \right\Vert_{0,\alpha} \lesssim 1.
$$
Using the expression (\ref{eq_Christoffel}) for the Christoffel symbols we see that
$$
d \max_{i,j,k} \left\Vert \Gamma^{k}_{ij} \right\Vert _{0,\alpha} \lesssim
\max_{i,j} \left\Vert g^{ij} \right\Vert_{0,\alpha} \cdot  
d \max_{i,j} \left\Vert D^{1}g_{ij} \right\Vert_{0,\alpha} \lesssim 1.
$$
Since $\partial_{l}g^{ij} = -g^{ia}g^{jb}\partial_{l}g_{ab}$ we obtain
$$
 d \max_{i,j} \left\Vert D^{1}g^{ij} \right\Vert_{0,\alpha} \lesssim
\left( \max_{i,j} \left\Vert g^{ij} \right\Vert_{0,\alpha} \right)^2 \cdot 
 d \max_{i,j} \left\Vert D^{1}g_{ij} \right\Vert_{0,\alpha}
\lesssim 1.
$$
By differentiating (\ref{eq_Christoffel}) we see that
\begin{align*}
d^{2} \max_{i,j,k} \left\Vert D^{1}\Gamma^{k}_{ij} \right\Vert_{0,\alpha} &\lesssim 
d \max_{i,j} \left\Vert D^{1}g^{ij} \right\Vert_{0,\alpha} 
\cdot
d \max_{i,j} \left\Vert D^{1}g_{ij} \right\Vert_{0,\alpha} 
\\ &+
\max_{i,j} \left\Vert g^{ij} \right\Vert_{0,\alpha} 
\cdot
d^{2} \max_{i,j} \left\Vert D^{2}g_{ij} \right\Vert_{0,\alpha} \lesssim 1.
\end{align*}
Hence according to (\ref{eq_gauss_curvature_metric})
together with the estimates above we obtain that
$$
d^{2} \left\Vert K \right\Vert_{0,\alpha} \lesssim
\left\Vert \frac{1}{g_{11}} \right\Vert_{0,\alpha} \cdot 
\left[ 
d^{2} \max_{i,j,k} \left\Vert D^{1}\Gamma_{ij}^{k} \right\Vert_{0,\alpha} + 
\left( d \max_{i,j,k} \left\Vert \Gamma_{ij}^{k} \right\Vert_{0,\alpha} \right)^{2} 
\right] \lesssim 1,
$$
so that (\ref{eq_1327}) holds true and the proof is completed.
\end{proof}

\begin{proof} [Proof of Theorem \ref{thm_1}]
Suppose that $M$ is extrinsically $\alpha$-regular at $p_{0} \in M$ to distance $\tilde{\delta} > 0$. According to Proposition \ref{prop_chart_def_scale} we have that $M$ is $\alpha$-regular via an $\eps$-restricted coordinate chart at $p_{0} \in M$ to distance $\delta =  \tilde{\delta} \eps / 14$. It therefore follows from Proposition \ref{prop_bound_curv_alpha_norm} that there exists a universal constant $C \geq 1$ such that  $\delta^{2} \Vert K \Vert_{0,\alpha;B} \leq C$. 
Then for $\delta_{0} = \delta/ \sqrt{C}$ and $B_{0} = B_{M}(p_{0},\delta_{0})$ we have that $B_{0} \subseteq B$ and
$$
\delta_{0}^{2} \left\Vert K \right\Vert_{0,\alpha;B_{0}} = 
C^{-1}\delta^{2} \left\Vert K \right\Vert_{0,\alpha;B_{0}} \leq
C^{-1}\delta^{2} \left\Vert K \right\Vert_{0,\alpha;B} \leq 1.
$$
Therefore we have that $\rho_{\text{int}} \geq (\eps / 14\sqrt{C}) \cdot \rho_{\text{ext}}$.
	
\end{proof}

\section{Apriori bounds for Poisson's equation}
\label{sec_claim_2}

For the proof of Theorem \ref{thm_2} we use an isothermal coordinate chart $z : \overline{B} \to \delta \overline{\DD}$, where $\delta \DD = B_{\RR^{2}}(0,\delta)$ stands for the Euclidean disc of radius $\delta$ around the origin. For the boundary of $\delta \DD$ we write $\delta \MBS^{1}$. 
In these coordinates the metric takes the form
$$
g = \varphi \cdot |dz|^{2},
$$ 
where $\varphi$ is a positive function, referred to as the {\it conformal factor}. The Gauss curvature function is given by Liouville's equation
\begin{equation}
K = -\frac{\Delta \left( \log \varphi \right)}{2 \varphi},
\label{eq_curvature_formula}
\end{equation}
where $\Delta$ is the usual Laplace operator with respect to the coordinate map $z = x+iy$. 
The following estimate from \cite{E} provides a bound for the supremum norm of the conformal factor $\varphi$ in terms of the curvature function $K$.
Recall that $d_{g}$ and $d_{g_{0}}$ stand for the Riemannian distance and the Euclidean distance respectively.

\begin{theorem} {\rm \cite[Corollary 1.2]{E}.}
	Let $M$ be a $C^{2}$-smooth Riemannian surface, fix $p_{0} \in M$ and let $\delta > 0$. 
	Suppose that at any point $p \in B = B_{M}(p_{0},\delta)$ the injectivity radius is at least $2\delta$ and $-\kappa \leq K(p) \leq \kappa$, where $\kappa > 0$ is a constant satisfying $\delta^{2}\kappa < \pi^{2} / 8$. 
	Let $z : \overline{B} \to \delta \overline{\DD}$ be an isothermal coordinate chart such that $z(p_{0}) = 0$, $z(\partial{B}) = \delta \MBS^{1}$ and whose conformal factor is $\varphi$. Then
	$$
	\sup_{B} \left| \log \varphi \right| \leq 8 \delta^{2}\kappa,
	$$
	and for any two distinct points $p, q \in B$ we have
	$$
	\exp(-4\delta^{2}\kappa)
	\leq \frac{d_{g}(p,q)}{d_{g_{0}}(z(p),z(q))} \leq 
	\exp(4\delta^{2}\kappa).
	$$
	\label{thm_isothermal}
\end{theorem}

\medskip
In order to relate the norms of the conformal factor $\varphi$ and the Gauss curvature function $K$ we apply a-priori bounds for a solution of Poisson's equation.
The definitions and propositions below are taken from \cite{GT}, sometimes using different notation.
Fix a bounded domain $\Omega \subseteq \RR^{n}$. For $x \in \Omega$ we write $d_{x} = \text{dist}(x,\partial{\Omega})$ for the distance between $x$ and the boundary of $\Omega$. 

\begin{proposition} {\rm (Gradient estimate \cite[Theorem 3.9]{GT}).}
Let $u \in C^{2}(\Omega)$ satisfy Poisson's equation $\Delta u = f$ in $\Omega$. Then
$$
\sup_{\Omega} d_{x}|\nabla u(x)| \leq C (\sup_{\Omega} |u| + \sup_{\Omega} d_{x}^{2}|f(x)|),
$$
where $C=C(n)$ is a constant depending on $n$.
\label{prop_simple_schauder}
\end{proposition} 

\medskip
In the following propositions we use an additional set of norms which are slightly different from the ones used in the previous sections.
For an integer $k\geq 0$, $\sigma \geq 0$ and $f \in C^{k}(\Omega)$ define
$$
\left[ f \right]^{(\sigma)}_{k,0;\Omega} =
\left[ f \right]^{(\sigma)}_{k;\Omega} = \sup_{x \in \Omega, |\beta| = k} d_{x}^{k+\sigma} \left| D^{\beta}f(x) \right|
\qquad \text{and} \qquad 
\left\Vert f \right\Vert^{(\sigma)}_{k;\Omega} =
\sum_{j=0}^{k} \left[ f \right]^{(\sigma)}_{j;\Omega}.
$$
For $0 < \alpha \leq 1$, denoting $d_{x,y} = \min(d_{x},d_{y})$, we also define
$$
[f]^{(\sigma)}_{k,\alpha;\Omega} = \sup_{x \neq y\in \Omega,|\beta|=k} d_{x,y}^{k+\alpha+\sigma} \frac{|D^{\beta}f(x)-D^{\beta}f(y)|}{|x-y|^{\alpha}}
\qquad \text{and} \qquad
\left\Vert f \right\Vert^{(\sigma)}_{k,\alpha;\Omega} =
\left\Vert f \right\Vert^{(\sigma)}_{k;\Omega} + 
\left[ f \right]^{(\sigma)}_{k,\alpha;\Omega}.
$$
In the case where $\sigma = 0$ we denote these quantities by $[\;\cdot\;]^{*} = [\;\cdot\;]^{(0)}$ and $\Vert \cdot \Vert^{*} = \Vert \cdot \Vert^{(0)}$.
Note that $\Vert \cdot\Vert_{k;\Omega}^{(\sigma)}$ and $\Vert \cdot \Vert_{k,\alpha;\Omega}^{(\sigma)}$ are norms on the subspaces of $C^{k}(\Omega)$ and $C^{k,\alpha}(\Omega)$ respectively for which they are finite. When the set $\Omega$ is fixed we sometimes omit it from our notation. 
It is not too difficult to verify that
\begin{equation}
\Vert fg \Vert^{(\sigma + \tau)}_{0,\alpha} \leq 
\Vert f \Vert^{(\sigma)}_{0,\alpha} \cdot
\Vert g \Vert^{(\tau)}_{0,\alpha}.
\label{eq_submul_astrix}
\end{equation}

\begin{lemma} {\rm(Interpolation inequalities \cite[Lemma 6.32]{GT}).}
Suppose $j+\beta < k + \alpha$, where $j,k \in \ZZ_{\geq 0}$ and $0\leq \alpha,\beta \leq 1$. Let $\Omega$ be an open subset of $\RR^{n}$ and assume $u\in C^{k,\alpha}(\Omega)$. Then for any $\eps > 0$ and some constant $C=C(\eps,k,j)$	we have
$$
\left[ u \right]^{*}_{j,\beta;\Omega} \leq 
C \left\Vert u \right\Vert_{0;\Omega} + 
\eps \left[ u \right]^{*}_{k,\alpha;\Omega}
\qquad \text{and} \qquad
\left\Vert u \right\Vert^{*}_{j,\beta;\Omega} \leq 
C \Vert u \Vert_{0;\Omega} + 
\eps \left[ u \right]^{*}_{k,\alpha;\Omega}.
$$
\label{lem_intepolation}
\end{lemma}

\begin{proposition} {\rm (Schauder estimate \cite[Theorem 4.8]{GT}).}
Let $u\in C^{2}(\Omega)$, $f\in C^{\alpha}(\Omega)$ satisfy $\Delta u = f$ in an open set $\Omega \subseteq \RR^{n}$. Then
$$
\left\Vert u \right\Vert^{*}_{2,\alpha;\Omega} \leq 
C \left( \left\Vert u \right\Vert_{0;\Omega} +
\left\Vert f \right\Vert^{(2)}_{0,\alpha;\Omega} \right),
$$
where $C=C(n,\alpha)$ is a constant depending on $n$ and $\alpha$.
\label{prop_schauder}	
\end{proposition}
 
\medskip
With the above estimates we are ready to prove Theorem \ref{thm_2}. In the following proposition we write $x \lesssim y$ if there exists a constant $C \geq 1$ depending on $\alpha$ such that $x \leq C \cdot y$. In case $x \lesssim y$ and $y \lesssim x$ we write $x \eqsim y$.
For a function $f$, an integer $j \geq 0$ and $\sigma \geq 0$ we use the notation $\Vert D^{j}f \Vert^{(\sigma)}_{k,\alpha} = \sup_{|\beta|=j} \Vert D^{\beta}f \Vert^{(\sigma)}_{k,\alpha}$ where $\beta$ is a two-dimensional multi-index.

\begin{proposition}
Let $0 < \alpha < 1$, let $M$ be a $C^{2}$-smooth Riemannian surface, fix a point $p_{0} \in M$ and let $\delta > 0$ be such that the injectivity radius of all points in the Riemannian disc $B = B_{M}(p_{0},\delta)$ is at least $2\delta$.
Suppose that $\delta^{2} \Vert K \Vert_{0,\alpha;B} \leq 0.01$, and let $\varphi$ be the conformal factor corresponding to an isothermal coordinate chart $z : \overline{B} \to \delta \overline{\DD}$ with $z(p_{0}) = 0$ and $z(\partial{B}) = \delta \MBS^{1}$. Then
$$
\left\Vert \varphi - 1 \right\Vert^{*}_{2,\alpha; \delta \DD} \lesssim 
 \delta^{2} \left\Vert K \right\Vert_{0,\alpha; B}.
$$
\label{prop_schauder_corollary}
\end{proposition}

\begin{proof}
According to Theorem \ref{thm_isothermal} for any $p \neq q \in B$ we have
$$
e^{- 0.04} \leq 
\frac{d_{g}(p,q)}{d_{g_{0}}(z(p),z(q))} 
\leq e^{0.04},
$$
which implies that
$ \Vert K \Vert_{0,\alpha;B} \eqsim \Vert K \Vert_{0,\alpha;\delta \DD}$.
Therefore all of the norms in this proof are implicitly with respect to $\delta \DD$, where the distances are Euclidean, and our goal is to show that 
$ \left\Vert \varphi - 1 \right\Vert^{*}_{2,\alpha} \lesssim \delta^{2} \left\Vert K \right\Vert_{0,\alpha}$.
According to Theorem \ref{thm_isothermal} together with the fact that $|e^{x} - 1| \leq 2|x|$ for $|x| < 1$ we see that 
$ \left\Vert \varphi - 1 \right\Vert_{0} \leq 2 \left\Vert \log \varphi \right\Vert_{0} \leq 16 \cdot \delta^{2} \left\Vert K \right\Vert _{0}$.
It therefore follows from Lemma \ref{lem_intepolation} that it suffices to show that
\begin{equation}
[\varphi]^{*}_{2,\alpha} \lesssim \delta^{2} \left\Vert K \right\Vert_{0,\alpha}.
\label{eq_0659}
\end{equation}
Using Liouville's equation (\ref{eq_curvature_formula}) together with Proposition \ref{prop_simple_schauder} and the fact that $ \Vert \varphi \Vert_{0} \lesssim 1$ we have
\begin{align*}
\sup_{x \in \delta \DD, |\beta| = 1} d_{x} \left| D^{\beta}\varphi(x) \right| 
& \lesssim
\sup_{x \in \delta \DD, |\beta| = 1} d_{x} \left| D^{\beta}\log \varphi(x) \right| 
\lesssim 
\sup_{\delta \DD} \left| \log\varphi \right| + 
\sup_{\delta \DD} d_{x}^{2} \left| K\varphi \right| 
\\& \lesssim
\sup_{\delta \DD} \left| \log\varphi \right| +
\delta^{2} \sup_{\delta \DD} \left| K \right| \sup_{\delta \DD} \left| \varphi \right| \lesssim 1.
\nonumber
\end{align*}
Hence by using Lemma \ref{lem_intepolation} we see that
$ \left\Vert \varphi \right\Vert^{*}_{0,\alpha} \lesssim \left\Vert \varphi \right\Vert^{*}_{1} \lesssim 1$.
According to Proposition \ref{prop_schauder} together with Liouville's equation (\ref{eq_curvature_formula}), Theorem \ref{thm_isothermal} and (\ref{eq_submul_astrix}) we therefore obtain that
\begin{align*}
\left\Vert \log \varphi \right\Vert^{*}_{2,\alpha} 
& \lesssim 
\left\Vert \log \varphi \right\Vert_{0} + 
\left\Vert K \varphi \right\Vert^{(2)}_{0,\alpha}  \leq
\left\Vert \log \varphi \right\Vert_{0} + 
\left\Vert K \right\Vert^{(2)}_{0,\alpha}  \left\Vert \varphi \right\Vert^{*}_{0,\alpha}
\\& \lesssim
\delta^{2} \left\Vert K \right\Vert_{0} + 
\delta^{2} \left\Vert K \right\Vert_{0,\alpha} \left\Vert \varphi \right\Vert^{*}_{0,\alpha} \lesssim \delta^{2} \left\Vert K \right\Vert_{0,\alpha}.
\label{eq_1710}
\end{align*}
For $i,j \in \{ 1,2\}$ we have that
$\partial_{i}\partial_{j}\varphi = 
\varphi \cdot (\partial_{i}\partial_{j}\log\varphi +
\partial_{i} \log\varphi \cdot \partial_{j} \log\varphi)$.
Moreover, it follows from the definitions and Lemma \ref{lem_intepolation} that for $k \in \{1,2 \}$ we have $\left\Vert D^{k}  \log \varphi \right\Vert^{(k)}_{0,\alpha}  \lesssim \left\Vert \log \varphi \right\Vert^{*}_{2,\alpha}$.
Therefore by using (\ref{eq_submul_astrix}) we obtain that
$$
[\varphi]^{*}_{2,\alpha} \leq 
\left\Vert D^{2}\varphi \right\Vert^{(2)}_{0,\alpha} \lesssim
\left\Vert D^{2}\log\varphi \right\Vert^{(2)}_{0,\alpha} + 
\left( \left\Vert D^{1} \log \varphi \right\Vert^{(1)}_{0,\alpha} \right)^{2} \lesssim
\delta^{2} \left\Vert K \right\Vert_{0,\alpha}.
$$
Hence (\ref{eq_0659}) holds true and the proof is completed.
\end{proof}

\begin{proof} [Proof of Theorem \ref{thm_2}]
Suppose that $M$ is intrinsically $\alpha$-regular at a point $p_{0} \in M $ to distance $\delta > 0$, i.e. the injectivity radius of all points of $B = B_{M}(p_{0},\delta)$ is at least $2\delta$ and 
$ \delta^{2} \left\Vert K \right\Vert_{0,\alpha;B} \leq 1$. 
Let $\delta_{0} = \delta /10 \sqrt{C_{0}}$ where $C_{0} \geq 1$ is the implied constant depending on $\alpha$ from Proposition \ref{prop_schauder_corollary}, and write $B_{0} = B_{M}(p_{0},\delta_{0})$. 
Then $\delta_{0}^{2} \Vert K \Vert_{0,\alpha;B_{0}} \leq 0.01 \cdot \delta^{2} \Vert K \Vert_{0,\alpha; B} \leq 0.01$.
Let $z_{0}:\overline{B_{0}} \to \delta_{0}\overline{\DD}$ be an isothermal coordinate chart such that $z_{0}(p_{0}) = 0$ and $z_{0}(\partial{B_{0}}) = \delta_{0} \MBS^{1}$ with conformal factor $\varphi_{0}$. According to Proposition \ref{prop_schauder_corollary} we have that
\begin{equation}
\left\Vert \varphi_{0} - 1 \right\Vert^{*}_{2,\alpha;\delta_{0}\DD} \leq C_{0} \delta_{0}^{2} \Vert K \Vert_{0,\alpha;B_{0}} \leq 
0.01 \cdot \delta^{2} \left\Vert K \right\Vert_{0,\alpha;B} \leq 0.01.
\label{eq_0223}
\end{equation}

\medskip
Consider $\delta_{1} = \delta_{0} / 3$, the Riemannian disc $B_{1} = B_{M}(p_{0},\delta_{1})$ and the isothermal coordinate map $z_{0} |_{B_{1}} : B_{1} \to U$, where $U \subseteq \delta_{0} \DD$ is some open set.
According to Theorem \ref{thm_isothermal} for any $q \in \partial{B_{1}}$ we have $|z_{0}(q)| \leq e^{0.04} \cdot \delta_{1} < \delta_{0}/2$.
Thus we see that $\text{dist}(U, \delta_{0} \MBS^{1}) > \delta_{0}/2$ and $d = \text{diam}(U) \leq \delta_{0}$. 
Writing $d_{x} = \text{dist}(x,\delta_{0} \MBS^{1})$ we obtain that $d \leq 2d_{x}$ for any $x \in U$. Hence
$$
d^{k} \left\Vert D^{k}\varphi_{0} \right\Vert_{0;U} \leq 
4 \cdot \sup_{x \in U,|\beta|=k} d_{x}^{k} \cdot |D^{\beta}\varphi_{0}(x)|
\leq 4 \cdot [\varphi_{0}]^{*}_{k,0;\delta_{0}\DD}
\qquad \text{for any } k \in \{1,2\}.
$$
Similarly, writing $d_{x,y} = \min(d_{x},d_{y})$,
$$
d^{2+\alpha}[D^{2}\varphi_{0}]_{\alpha;U} \leq
2^{2+\alpha} \cdot \sup_{x \neq y \in U,|\beta|=2} d_{x,y}^{2+\alpha} \frac{|D^{\beta}\varphi_{0}(x)-D^{\beta}\varphi_{0}(y)|}{|x-y|^{\alpha}}
\leq 
2^{2+\alpha} \cdot [\varphi_{0}]^{*}_{2,\alpha;\delta_{0}\DD}.
$$
Using (\ref{eq_0223}) we therefore obtain that
$ \left\Vert \varphi_{0} - 1 \right\Vert_{2,\alpha;U} \leq 8 \cdot \left\Vert \varphi_{0} - 1 \right\Vert^{*}_{2,\alpha;\delta_{0}\DD} \leq 0.08$. 
All that remains is to ``fix'' the value of the conformal factor at the origin, i.e. to make it equal to one.
Let $\xi = \sqrt{\varphi_{0}(0)}$, $V = \xi U$ and set $z_{1} = \xi z_{0}|_{B_{1}} : B_{1} \to V$. The isothermal coordinate map $z_{1}$ has the conformal factor $\varphi_{1}(z) = \xi^{-2}\varphi_{0}(z/\xi) $, and
\begin{align}
\left\Vert \varphi_{1} - 1 \right\Vert_{2,\alpha;V} &\leq 
\left\Vert \varphi_{1} - \xi^{-2}\right\Vert_{2,\alpha;V} + 
\left| 1-\xi^{-2} \right|  = 
\xi^{-2} \left\Vert \varphi_{0} - 1 \right\Vert_{2,\alpha;U} + 
\left| 1-\xi^{-2} \right| \nonumber \\&\leq 
0.08 \cdot \xi^{-2} + |1 - \xi^{-2}| < 1 .
\label{eq_1708}
\end{align}
Hence by taking $C = 1 /30\sqrt{C_{0}}$, which is a constant depending on $\alpha$, we have that $\delta_{1} = C\delta$ and the Riemannian disc $B_{M}(p_{0},\delta_{1})$ is isometric to a bounded open set $V \subseteq \RR^{2}$ endowed with the metric $\varphi_{1}|dz_{1}|^{2}$ where $\varphi_{1}(0) = 1$ and (\ref{eq_1708}) holds true. Therefore $\rho_{\text{ext}} \geq C \cdot \rho_{\text{int}}$ and the proof is completed.
\end{proof}

\medskip
\noindent Department of Mathematics, Weizmann Institute of Science, Rehovot 76100, Israel. \\
{\it e-mail:} \verb"matan.eilat@weizmann.ac.il"

\end{document}